\documentclass[psamsfonts,11pt]{amsart}

\usepackage{amssymb,amsfonts,tikz,mathtools,amsthm, amsrefs}
\usepackage[all,arc]{xy}
\usepackage{enumerate}
\usepackage{mathrsfs}
\usepackage{setspace}
\newtheorem{thm}{Theorem}[section]
\newtheorem{cor}[thm]{Corollary}
\newtheorem{prop}[thm]{Proposition}
\newtheorem{lem}[thm]{Lemma}
\newtheorem{conj}[thm]{Conjecture}

\theoremstyle{definition}
\newtheorem{defn}[thm]{Definition}

\theoremstyle{remark}

\newcommand{\bR}{{\mathbb R}}
\newcommand{\bZ}{{\mathbb Z}}
\newcommand{\bQ}{{\mathbb Q}}
\newcommand{\bC}{{\mathbb C}}
\newcommand{\bT}{{\mathbb T}}
\newcommand{\bN}{{\mathbb N}}

\renewcommand{\d}{{\mathrm d}}
\newcommand{\rmm}{{\mathrm m}}

\renewcommand{\th}{{\widetilde{h}}}
\newcommand{\tH}{{\widetilde{H}}}

\newcommand{\bfx}{{\mathbf{x}}}

\newcommand{\hide}[1]{}

\numberwithin{equation}{section}

\title[M\"{o}bius Disjointness for Skew Products on Heisenberg
Nilmanifold]{M\"{o}bius Disjointness for Skew Products on the Heisenberg
Nilmanifold}

\author{Matthew Litman}
\address{University of California - Davis, Davis, CA 95616, USA}
\email{mclitman@ucdavis.edu}

\author{Zhiren Wang}
\address{Pennsylvania State University,
University Park, PA 16801, USA}
\email{zhirenw@psu.edu}

\thanks{This paper was the outcome of an undergraduate research project sponsored by the Eberly College of Science at Penn State University during the 2016-2017 academic year. M.L. thanks the ECoS for its support. Z.W., the faculty mentor of the project, was supported by the NSF grant DMS-1501295.}

\date{\today}

\begin{document}

\maketitle

\begin{abstract} We prove that the M\"obius function is disjoint to all Lipschitz continuous skew product dynamical systems on the 3-dimensional Heisenberg nilmanifold over a minimal rotation of the 2-dimensional torus.
\end{abstract}

\section{Introduction}

\subsection{Setting and statement} The M\"obius function $\mu:\bN\to\{-1,0,1\}$ is defined as follows: $\mu(n)=(-1)^k$ if $n$ is the product of $k$ distinct primes, and $\mu(n)=0$ otherwise. Sarnak's M\"obius disjointness conjecture states that $\mu(n)$ is highly random, in the sense that it is orthogonal to all continuous observables from zero-entropy topological dynamical systems. In this article, we deal with a special case of this conjecture, namely Lipschitz continuous skew product maps on the 3-dimensional Heisenberg nilmanifold.

The Heisenberg group is 
\begin{equation}\label{EqGroup}G=\{(x,y,z):x,y,z\in\mathbb{R}\}\cong\mathbb{R}^3\end{equation} equipped with the group rule \begin{equation}\label{EqGroupRule}(x,y,z)(x',y',z')=(x+x',y+y',z+z'+(xy'-x'y)).\end{equation} Set $\Gamma = G(\mathbb{Z})=\{(x,y,z)\in G: x,y,z\in\bZ\}$ and $X=G/\Gamma$. Then $X$ is a compact nilmanifold and its maximal torus factor is $\bT^2=\bR^2/\bZ^2$, parametrized by the $x$ and $y$ coordinates. $X$ is a principal $\bT^1$-bundle over $\bT^2$. $G$ acts on $X$ by left translation.

For $\alpha,\beta\in\bR$ and a continuous function $h:\mathbb{T}^2 \to \bT^1$, define $T:X \to X$ by 
\begin{equation}\label{EqMap}\bfx\mapsto (\alpha, \beta, \th(x,y))\bfx,\end{equation} where $\bfx= (x,y,z) \Gamma$, $\th(x,y)$ is any lifting of the value $h(x,y)\in\bT^1$ to $\bR$, and $(\alpha, \beta, \th(x,y))$ stands for an element in $G$. Here we regard $h$ as a $\bZ^2$-periodic function on $\bR^2$. 

Indeed, the choice of $\th(x,y)$ does not matter. This is because for two different choices of $\th(x,y)$, the values of $(\alpha, \beta, \th(x,y))$ differ by translation by an element from the group $C=\{(0,0,m):m\in\bZ\}$. This group is both in the center of $G$ and in $\Gamma$, so the two different choices of $(\alpha, \beta, \th(x,y))\bfx$ represent the same point in $X=G/\Gamma$.

Without causing confusion, we will simply write \eqref{EqMap} as
\begin{equation}\label{EqMap2}T:\bfx\mapsto (\alpha, \beta, h(x,y))\bfx.\end{equation} Here $(\alpha, \beta, h(x,y))$ should be think of as an element in the quotient group $G/C$.

The map $T$ is an isometric extension of the translation by $(\alpha,\beta)$ on $\bT^2$, which we denote by $T_0$. Namely, $T_0$ is a factor of $T$, and $T$ send fibers (which are circles $\bT^1$) to fibers by isometries. In particular, $(X,T)$ is a distal dynamical system and has zero topological entropy.

Recall that $T_0$ is minimal and ergodic on $\bT^2$ if $\alpha$, $\beta$, $1$ are linearly independent over $\bQ$. Otherwise, every orbit of $T_0$ is contained in a finite union of parallel $1$-dimensional subtori in $\bT^2$.

Our main result is:

\begin{thm}\label{ThmMain} If $\alpha$, $\beta$, $1$ are linearly independent over $\bQ$ and $h:\bT^2\to\bT^1$ is Lipschitz continuous, then \begin{equation}\label{EqSarnak}\lim_{N\to\infty}\frac1N\sum_{n=1}^Nf(T^n\bfx)\mu(n)=0,\ \forall\bfx\in X,\forall f\in C(X).\end{equation}\end{thm}

We remark that the assumption on $\alpha$, $\beta$ is in place only to guarantee minimality, and no extra Diophantine conditions are needed.

\subsection{Background and motivation} The M\"obius disjointness conjecture, proposed by Sarnak \cite{S09}, is:

\begin{conj}\label{ConjSarnak} For a topological dynamical system $(X,T)$, if $h_{\mathrm{top}}(T)=0$, then \eqref{EqSarnak} holds.\end{conj}

The conjecture has been the subject of many recent researches. For known cases of the conjecture, see \cites{B13a, B13b, BSZ13, D37, DK15, EKL16, EKLR15, ELD14, ELD15, FJ15, FFKL16, FKLM15, G12, GT12, HLSY15, HWZ16, KL15, LS15, MMR14, MR10, MR15, M16, P15, V16, W17}, to list a few.

An important class of zero entropy topological dynamical systems are distal dynamical systems. By Furstenberg's structure theorem \cite{F63}, minimal distal systems are inverse limits of towers of isometric extensions.

M\"obius disjointness for homogeneous distal dynamical systems were known by the works of Davenport \cite{D37} for rotations of the circle, of Green-Tao \cite{GT12} for nilflows, and of Liu-Sarnak \cite{LS15} for all affine distal flows.

According to Furstenberg's structure theorem, the simplest non-homogeneous distal systems are 2-step isometric extensions, i.e. an isometric extension of a rotation on a compact abelian group. 

For manifolds, $\bT^2$ is the smallest on which one can create such a map, which is the skew product $T(x,y)=(x+\alpha, y+h(x))$. M\"obius disjointness for such skew products is proved for generic $\alpha$ when $h$ is $C^{1+\epsilon}$ by Ku\l{}aga-Pryzmus and Lemanczyk \cite{KL15}, as well as for all $\alpha$ when $T$ is analytic by Liu and Sarnak \cite{LS15} and Wang \cite{W17}. 

The aim of this paper is to demonstrate that the problem is easier to handle for non-homogeneous dynamical systems when the isometric extension's underlying fiber bundle structure is not trivial.

In the settings of Theorem \ref{ThmMain}, the Heisenberg nilmanifold is a non-trivial principal circle bundle over $\bT^2$. The twistedness of the topology allows to show unique ergodicity of a dynamical system that is induced from $T$ using the K\'atai-Bourgain-Sarnak-Ziegler criterion \cites{K86, BSZ13}, assuming Lipschitz continuity. In contrast, for skew products on $\bT^2$, which is a trivial circle bundle over the circle, the works \cite{LS15} and \cite{W17} required either methods from harmonic analysis or Matom\"aki-Radziwi\l{}\l{}-Tao bounds \cite{MRT15} on short averages of multiplicative functions, in addition to the K\'atai-Bourgain-Sarnak-Ziegler criterion, and needed to assume analyticity.

We remark that the proof of Theorem \ref{ThmMain} can be easily extended to skew products on higher dimensional Heisenberg manifolds and other 2-step nilmanifolds. However, we are not going to pursue this direction in detail.\\

\noindent{\it Notations.} On $\bT^1=\bR/\bZ$, $\|\cdot\|$ denotes the distance to the origin. The function $e(\cdot)$ on $\bT^1$ (or $\bR$) is defined as $e(x)=e^{2\pi ix}$. For a compact nilmanifold or torus $Y$, $\rmm_Y$ denotes the unique uniform probability measure on $Y$, which descends from a Haar measure on the universal cover of $Y$.\\

\section{Proof of Theorem \ref{ThmMain}}

\subsection{Reduction of the joining dynamics}\label{SecNonUE}

We suppose $\bfx_0\in X$ and a function $f_0\in C(X)$ does not satisfy \eqref{EqSarnak}. By translating both the point and the function,  we may assume without loss of generality that $\bfx_0$ is the identity point $\Gamma$ in $X=G/\Gamma$, i.e. represented by $(0,0,0)$.

\begin{defn}A continuous function $f: X\to\bC$ has vertical oscillation of frequency $\xi \in\bZ$ if for all $\tau\in X$ and $z\in\bR$, $$f((0,0,z)\tau)=e(\xi z)f(\tau).$$\end{defn}

\begin{lem}\label{LemVertical}There exists a non-zero integer $\xi$ and a continuous function $f: X\to\bC$ of vertical oscillation of frequency $\xi$, such that \begin{equation}\label{EqLemVertical}\frac1N\sum_{i=1}^Nf(T^n\bfx_0)\mu(n)\not\to 0\text{ as }N\to\infty.\end{equation}\end{lem}
\begin{proof}By our earlier hypothesis, there are $\delta\in(0,1)$ and a subsequence $\{N_i\}$ of $\bN$, such that $$\left|\frac1{N_i}\sum_{n=1}^{N_i}f_0(T^n\bfx_0)\mu(n)\right|>\delta.$$ On the other hand, by the proof of \cite{GT12a}*{Lemma 3.7}, there are finitely many continuous functions $f_j$, $1\leq j\leq J$ on $X$ of vertical oscillation, respectively of frequecy $\xi_j$, such that $\|f_0-\sum_{j=1}^Jf_j\|_{L^\infty}<\frac\delta2$. It follows that
$$
\left|\frac1{N_i}\sum_{n=1}^{N_i}f_0(T^n\bfx_0)\mu(n)-\sum_{j=1}^J\frac1{N_i}\sum_{n=1}^{N_i}f_j(T^{pn}\bfx_0)\mu(n)\right|<\frac\delta2.
$$ and hence $$
\left|\sum_{j=1}^J\frac1{N_i}\sum_{n=1}^{N_i}f_j(T^{pn}\bfx_0)\mu(n)\right|>\delta-\frac\delta2=\frac\delta2.
$$ for all $i$. In other words, $\sum_{j=1}^J\frac1N\sum_{n=1}^Nf_j(T^{pn}\bfx_0)\mu(n)\not\to 0$ as $N\to\infty$. We deduce that for at least one $j$, $\frac1N\sum_{n=1}^Nf_j(T^{pn}\bfx_0)\mu(n)\not\to 0$. Let $f=f_j$ and $\xi=\xi_j$. Then \eqref{EqLemVertical} holds.

It remains to claim that $\xi\neq 0$. Indeed, if $\xi=0$, then $f$ is constant under translations along the vertical subgroup $\{(0,0,z)\}$, which are fibers of $X\to\bT^2$. Equivalently, $f$ can be thought of as a continuous function on $\bT^2$, and \eqref{EqLemVertical} can be rewritten as $$\frac1N\sum_{i=1}^Nf(T_0^n(0,0))\mu(n)\not\to 0\text{ as }N\to\infty.$$ As $T_0$ is the translation by $(\alpha,\beta)$ on $\bT^2$, this contradicts Davenport's theorem \cite{D37}. So we conclude that $\xi\neq 0$.\end{proof}

The following important criterion guarantees M\"obius disjointness and is due to K\'atai \cite{K86} and Bourgain-Sarnak-Ziegler \cite{BSZ13}:

\begin{thm}For a dynamical system $(\mathcal X, T)$, a continuous function $f\in C(\mathcal X)$, and a point $x\in\mathcal X$, if the equation \eqref{EqSarnak} fails to hold, then there exist a pair of distinct primes $p>q$, such that $\frac1N\sum_{n=1}^Nf(T^{pn}x)\overline{f(T^{qn}x)}$ does not converge to $0$ as $N\to\infty$.\end{thm}

By this criterion, for a pair of distinct primes $p>q$, 
\begin{equation}\label{EqBilinear}\frac1N\sum_{n=1}^Nf(T^{pn}\bfx_0)\overline{f(T^{qn}\bfx_0)}\not\to 0 \text{ as }N\to\infty.\end{equation}

We study the dynamics of the pair $(T^{pn}\bfx_0, T^{qn}\bfx_0)$.

\begin{lem}\label{LemInvSubgp}
The set $G_1 = \{(x_1,y_1,z_1,x_2,y_2,z_2) \ | \ q(x_1,y_1)=p(x_2,y_2) \}  \subseteq G^2$ is a closed subgroup of $G^2$ with the following properties:
\begin{enumerate}
\item $G_1 / \Gamma_1$, where $\Gamma_1 = G_1 \cap (\Gamma \times \Gamma
)$, is compact.
\item $(T^{pn}\bfx_0, T^{qn}\bfx_0) \in X_1 = G_1 / \Gamma_1$ for all $n$.
\end{enumerate}
\end{lem}

\begin{proof}(i) The nilpotent group $G$ is (the real points of) an algebraic group defined over $\bQ$ and thus so is $G^2$. The lattice $\Gamma$ is given by $G(\bZ)$. In order to show that $\Gamma_1$ is cocompact in $G_1$, it suffices to prove $G_1$ is a subgroup defined over $\bQ$. This is true by definition.

(ii) Notice that $(\bfx_0,\bfx_0)$ is the identity element in $X^2=G^2/\Gamma^2$. It suffices to show that the embedded subnilmanifold $X_1\subset X^2$ is $T^p\times T^q$-invariant. This can be verified from the definition of $G_1$, because $T^p$ adds $(p\alpha,p\beta)$ to the coordinate pair $(x_1,y_1)$ and  $T^q$ adds $(q\alpha,q\beta)$ to $(x_2,y_2)$.\end{proof}

\begin{lem}
For $D = \{(0,0,z_1,0,0,z_2 \ | \ z_1 = z_2 \} \subset G_1$ and $G_\ast = G_1 / D$, the subgroup $\Gamma_\ast = \Gamma_1 / \Gamma_1 \cap D$ is a cocompact lattice in $G_\ast$, and thus $X_\ast = G_\ast / \Gamma_\ast = X_1 / (D / \Gamma_1 \cap D)$ is a compact nilmanifold.
\end{lem}
\begin{proof}Remark first that $D$ is in the center of $G_1$, so $G_\ast$ is a group. Again, it suffices to notice that $D$ is an algebraic subgroup of the nilpotent group $G_1$ defined over $\bQ$.\end{proof}

We now describe the natural projection from $X_1$ to $X_\ast$. Because of the definition of $G_1$, each point in $G_1$ can be uniquely written as $(px,py,z_1,qx,qy,z_2)\in G^2$ for some $x,y,z_1,z_2\in\bR$, where $G$ is parametrized as in \eqref{EqGroup}. The $D$-orbit of this point is the set $\{(px,py,z_1+a,qx,qy,z_2+a):a\in\bR\}$. So $G_\ast=G_1/D$ can be parametrized by $\{(x,y,z):x,y,z\in\bR\}$, and the projection $\pi:G_1\to G_\ast$ is given by
\begin{equation}\label{EqProjection}\pi(px,py,z_1,qx,qy,z_2) = (x,y,z_1-z_2).\end{equation}

Because $p,q$ are distinct primes, each point in $\Gamma_1=G_1\cap\Gamma$ can be uniquely written as $(px,py,z_1,qx,qy,z_2)\in G^2$ for some $x,y,z_1,z_2\in\bZ$. Combining this with \eqref{EqProjection}, we see that $\Gamma_\ast=\pi(\Gamma_1)$ is just the set of integer points $\{(x,y,z)\in G_\ast:x,y,z\in\bZ\}$ of $G_\ast$.

\begin{lem}\label{LemNewGroupRule}
The group rule in $G_\ast$, which we denote by $\ast$, is given by 
\begin{align*}
(x,y,z) \ast (x',y',z') = (x+x', y+y', z+z' + (p^2-q^2)(xy'-x'y)).
\end{align*}
\end{lem}
\begin{proof}
The group rule in $G_1$ is
\begin{align*}
(px,&py,z_1,qx,qy,z_2)(px',py',z_1',qx',qy',z_2') = \\
&(p(x+x'),p(y+y'),z_1+z_1'+p^2(xy'-x'y),\\
&q(x+x'),q(y+y'),z_2+z_2'+q^2(xy'-x'y)).
\end{align*}
Applying $\pi$ to both sides, we get the formula in the lemma.
\end{proof}

It is not hard to see that the 2-step compact nilmanifold $X_\ast=G_\ast/\Gamma_\ast$, similar to the Heisenberg nilmanifold $X=G/\Gamma$, is a principal $\bT^1$-bundle over $\bT^2$. The base $\bT^2$ is parametrized by the first two coordinates $(x,y)$.

We indifferently denote by $\pi$ the projection from $X_1$ to $X_\ast$, which is induced from $\pi:G_1\to G_\ast$. By Lemma \ref{LemInvSubgp}, for all $n$ we have a point $\pi (T^{pn}\bfx,T^{qn}\bfx) \in X_\ast$.

The group $G_\ast$ acts by left translation on $X_\ast/\Gamma_\ast$. We keep the symbol $\ast$ to denote this action. It should be noted that, as $\pi: G_1\to G_\ast$ is a group morphism, for $g\in G_1$ and $\overline{\bfx}\in X_1$, $\pi g\ast\pi \overline{\bfx}=\pi(g\ast\overline{\bfx})$.

To proceed, we will need an expression for the $n$-th iterate $T^n$ for $n\in\bN$.
\begin{lem}\label{LemIterate} Let 
$h_n(x,y)=\sum_{i=0}^{n-1}h(x+i\alpha,y+i\beta)$. Then for $\bfx=(x,y,z)\Gamma\in X$ and $n\in\bN$, $$T^n\bfx=\big(x+n\alpha,y+n\beta,h_n(x,y)\big)\bfx.$$\end{lem}

\begin{proof}Because $T$ factors to $T_0$ on $\bT^2$, the projection of $T^n\bfx$  to $\bT^2$ is represented by $(x+n\alpha,y+n\beta)$. Thus $T^{n+1}\bfx=\big(\alpha,\beta,h(x+n\alpha,y+n\beta)\big)\cdot T^n\bfx$.

When $n=0$, the equality in the lemma automatically holds as $h_0(x,y)=0$. Suppose the lemma is true for $n$, then \begin{align*}&T^{n+1}\bfx\\=&\big(\alpha,\beta,h(x+n\alpha,y+n\beta)\big)\big(n\alpha,n\beta,h_n(x,y)\big)\bfx\\
=&\big((n+1)\alpha,(n+1)\beta,h_n(x,y)+h(x+n\alpha,y+n\beta)+\alpha\cdot n\beta-\beta\cdot n\alpha\big)\bfx\\
=&\big((n+1)\alpha,(n+1)\beta,h_{n+1}(x,y)\big)\bfx\end{align*} by the group rule \eqref{EqGroupRule}. This establishes the lemma by induction.\end{proof}

Given the functions $h_n$ in Lemma \ref{LemIterate}, we can define a piecewise continuous function $H:\bT^2\mapsto\bT^1$ by \begin{equation} H(x,y)=h_p(px,py) - h_q(qx,qy)\end{equation} on $\bT^2$.

\begin{cor}\label{CorJoiningDyn}
$\pi\circ(T^p \times T^q) = T_\ast \circ \pi$, where $$T_\ast \bfx_\ast = (\alpha , \beta , H(x,y)) \ast \bfx_\ast$$ if $\bfx_\ast\in X_\ast$ is the equivalence class containing $(x,y,z)\in G_\ast$.
\end{cor}

We remark that here, as in \eqref{EqMap2}, $(\alpha,\beta,H(x,y))$ should be viewed as an element of the group $G_\ast/C_\ast$ where $C_\ast=\{(0,0,m)\in G_\ast: m\in\bZ\}$. For different choices of $\tH(x,y)\in\bR$ lifting $H(x,y)\in\bT^1$, $(\alpha , \beta , \tH(x,y))$ differ by a defect in $C_\ast$. As $C_\ast$ is both in the cetner of $G_\ast$ and in $\Gamma_\ast$, this defect does not affect the position of $(\alpha , \beta , \tH(x,y)) \ast \bfx_\ast$. So we can write $H$ instead of $\tH$ in Corollary \ref{CorJoiningDyn}.

\begin{proof} Suppose $\overline{\bfx}\in X_1$ is represented by $(px,py,z_1,qx,qy,z_2)\in G_1$. By Lemma \ref{LemIterate} and formula \eqref{EqProjection}
\begin{align*}
\pi((T^p \times T^q)\overline{\bfx} ) =& \pi ((p\alpha, p \beta, h_p(px,py), q \alpha, q \beta, h_q(qx,qy))\cdot\overline{\bfx}) \\
=& \pi ((p\alpha, p \beta, h_p(px,py), q \alpha, q \beta, h_q(qx,qy))) \ast \pi \overline{\bfx} \\
=& (\alpha , \beta ,  H(x,y)) \ast \pi \overline{\bfx}.
\end{align*} The corollary is proved.
\end{proof}

Note that $T_\ast$ is a skew product map on $X_\ast$. It also descends to $T_0$ on $\bT^2$, and acts by rotations along the fiber direction. Hence, $T_\ast$ preserves the uniform probability measure $\rmm_{T_{X_\ast}}$.

We define $f_1$ on $X^2=G^2/ \Gamma^2$ (and thus on $X_1 \subseteq X^2$) by $f_1(\bfx_1 , \bfx_2) = f(\bfx_1)\overline{f}(\bfx_2)$. Because $f$ has vertical oscillation of frequency $\xi$, $f_1$ is invariant by $D$. Thus $f_1$ descends to a function $f_\ast$ on $X_\ast$.

\begin{lem}\label{LemZeroAvg} $\int_{X_\ast} f_\ast\d\rmm_{X_\ast}= 0$.\end{lem}
\begin{proof} Since $\xi \neq 0$, we have that
\begin{align*}&\int_{X_1} f_1\d\rmm_{X_1}\\
=&\int_{x,y,z_1,z_2\in[0,1)} f_1\big((px,py,z_1,qx,qy,z_2)\Gamma^2\big)\d x\d y\d z_1\d z_2\\
=&\int_{x,y\in[0,1)}\Big(\int_0^1f\big((px,py,z_1)\Gamma\big)\d z_1\Big)\Big(\int_0^1\overline{f}\big((qx,qy,z_2)\Gamma\big)\d z_2\Big)\d x\d y\\
=&\int_{x,y\in[0,1)}0\cdot 0\d x\d y =0.\end{align*} This implies the lemma, as $f_1$ and $\rmm_{X_1}$ respectively descend to $f_\ast$ and $\rmm_{X_\ast}$.\end{proof}

Let $\bfx_{0*}$ be the identity point $(0,0,0)\ast\Gamma_\ast$ in $X_\ast$. By Lemma \ref{LemInvSubgp} and Corollary \ref{CorJoiningDyn}, the average in \eqref{EqBilinear} can be formulated as 
\begin{equation}\label{EqProjectAvg}
\frac{1}{N}\sum_{n=1}^N f(T^{pn}\bfx_0) \overline{f}(T^{qn}\bfx_0) = 
\frac{1}{N}\sum_{n=1}^N f_\ast(T_\ast^n \pi(\bfx_0,\bfx_0))= 
\frac{1}{N}\sum_{n=1}^N f_\ast(T_\ast^n\bfx_{0*}).
\end{equation}

From this, we can conclude the analysis above by stating the following proposition.
\begin{prop}\label{PropNonUE}Under the hypotheses of this section, $T_\ast$ is not uniquely ergodic.\end{prop}
\begin{proof}If $T_\ast$ is uniquely ergodic, its unique invariant probability measure must be $\rmm_{X_\ast}$. Then by Birkhoff ergodic theorem and Lemma \ref{LemZeroAvg}, the ergodic averages $\frac{1}{N}\sum_{n=1}^N f_\ast(T_\ast^n\omega_{0*})$ converges to $0$. This contradicts \eqref{EqBilinear}, because of \eqref{EqProjectAvg}.\end{proof}

\subsection{Unique ergodicity of the reduced joining dynamics} By the proposition above, in order to prove Theorem \ref{ThmMain}, it suffices to show
\begin{prop}\label{PropUE} $T_\ast$ is uniquely ergodic.\end{prop}

In \cite{F61}, Furstenberg proved that the unique ergodicity for a skew product map on a circle bundle over a uniquely ergodic base that acts as rotations on the fibers is equivalent to the non-existence of invariant multi-valued graphs. He originally stated this criterion for skew products generated by a continuous cocycle. The same proof also works for measurable cocycles, which is the statement we will need (Theorem \ref{ThmFurst} below). For completeness' sake, we include the proof here.

\begin{thm}\label{ThmFurst}
Let $(\Omega_0,T_0)$ be a uniquely ergodic topological dynamical system, whose unique invariant probability measure is denoted by $\gamma_0$. Take $\Omega=\Omega_0\times \bT^1$ and define a skew product map $T:\Omega \to \Omega$ by $T(\omega_0,\zeta)=(T_0 \omega_0,g(\omega_0)+\zeta)$, where $g:\Omega_0\to \bT^1$ is a measurable function. Then:
\begin{enumerate}
\item The product measure $\gamma=\gamma_0\times\rmm_{\bT^1}$ is an invariant probability measure for $T$;
\item $T$ is uniquely ergodic if and only if for all $k\in \bN$, the equation 
\begin{equation}
\label{g condition}
R(T_0 \omega_0)=R(\omega_0)+k g (\omega_0)
\end{equation}
has no measurable solution $R:\Omega_0\to \bT^1$ modulo $\gamma_0$.
\end{enumerate}
\end{thm}

\begin{proof} The proof of Part (i) is straightforward, so we only discuss the second part.

The key claim is:

{\it\hskip2cm $T$ is uniquely ergodic if and only if $\gamma$ is ergodic.}

 To see this, define the transformation $\tau_\beta : \Omega \to \Omega$ by $\tau_\beta (\omega_0,\zeta)=(\omega_0,\beta+\zeta)$. Since $\gamma=\gamma_0\times\rmm_{\bT^1}$, if $\omega_\ast$ is a generic point for $(\Omega, T, \gamma)$, in the sense that $\frac1N\sum_{i=0}^{N-1}\delta_{T^i\omega_\ast}\to\gamma$ in the weak-$^*$ topology as $N\to\infty$, then so is $\tau_\beta (\omega_\ast)$ for every $\beta \in \bT^1$.

To show ergodicity implies unique ergodicity, suppose $T$ is ergodic with respect to $\gamma$. It follows that almost all points of $\Omega$ (with respect to $\gamma$) are generic for $(\Omega, T, \gamma)$. So $\gamma_0$-almost every $\omega_0\in\Omega_0$ has the property that $(\omega_0,\zeta)$ is generic for $(\Omega, T, \gamma)$ for $\rmm_{\bT^1}$-almost every $\zeta$. By applying $\tau_\beta$ for all $\beta\in \bT^1$, we see that for $\gamma_0$-almost every $\omega_0\in\Omega_0$,  $(\omega_0,\zeta)$ is generic for $(\Omega, T, \gamma)$ for all $\zeta\in \bT^1$.  Suppose $T$ is not uniquely ergodic, then there exists an ergodic probability measure $\gamma'$ other than $\gamma$ for $T$. As any $T$-invariant measure on $\Omega$ projects to an invariant measure on $\Omega_0$, and $(\Omega_0,T_0,\gamma_0)$ is uniquely ergodic, it follows that the projection of $\gamma'$ on $\Omega_0$ is $\gamma_0$. Thus for $\gamma_0$-almost all points $\omega_0$ in the base $\Omega_0$, there exist extended points $(\omega_0,\zeta)$ that are generic for $(\Omega, T, \gamma')$. This cannot happen though, since for almost every $\omega_0$ and all $\zeta$, $(\omega_0,\zeta)$ is generic for $\gamma$, which is different from $\gamma'$. This establishes the claim.

It remains to show that the ergodicity of $\gamma$ is equivalent to the condition in part (ii). 

Suppose first that $\gamma$ is not ergodic. Then $Tf=f$ has a non-constant solution $f\in L^2(\Omega, \gamma)$. Since $\gamma=\gamma_0\times m_{\mathbb{T}^1}$ is a product and $f$ is $L^2$ with respect to $\gamma$, we can split $f$ into Fourier series along the $\bT^1$ direction and write it as $$f= \sum_{-\infty}^\infty c_k(\omega_0)e(k\zeta),$$ where $c_k(\omega_0)\in L^2(\Omega_0,\gamma_0)$ and $e(\xi)=e^{2\pi i \xi}$. The condition $Tf=f$ implies $\sum_{-\infty}^\infty c_k(T_0\omega_\ast)e(kg(\omega_0)+k\zeta)=\sum_{-\infty}^\infty c_k(\omega_0)e(k\zeta)$, or \begin{equation}\label{EqPfFurst1}c_k(T_0\omega_0)e(kg(\omega_0))=c_k(\omega_0)\end{equation} for every $k\in\bZ$.

Since $T_0$ is ergodic, $f$ is not reducible to a function of $\omega_0$ alone and thus $c_k(\omega_0)\neq 0$ for at least one non-zero integer $k$. By the ergodicity of $T_0$ it follows that $c_k$ vanishes only on a set of measure zero, which allows us to write $c_k(\omega_0)$ as $r_k(\omega_0)e(\theta_k(\omega_0))$, where $r_k(\omega_0)>0$ and $\theta_k(\omega_0)\in \bT^1$. From \eqref{EqPfFurst1}, we get that $r_k(T_0\omega_0)e(\theta_k(T_0\omega_0)+kg(\omega_0))=r_k(\omega_0)e(\theta_k(\omega_0))$ for every $k$, thus $R(\omega_0)=-\theta_k(\omega_0)$ is a solution to \eqref{g condition}. In addition, if $k<0$, then we can replace $k$ with $-k$ and $R$ with $-R$. So one can claim $k\in\bN$ without loss of generality.

Conversely, if \eqref{g condition} has a solution, then the non-constant measurable function $e(-k\zeta)e(R(\omega_0))$ is invariant under $T$ modulo $\gamma$, implying that $\gamma$ is not ergodic, and we are done.\end{proof}

We now reparametrize $X_*$ in a piecewise continuous way in order to identify it with $\bT^3=\bT^2\times \bT^1$ and apply Theorem \ref{ThmFurst}.

In the parametrization given by Lemma \ref{LemNewGroupRule}, the box $[0,1)^3$ is a fundamental domain for the projection $G_\ast\to X_\ast$. Indeed, for each $(x,y,z)\in G_\ast$, there is a unique element of $\Gamma_\ast$, which we denote by $\lfloor(x,y,z)\rfloor$, such that $(x,y,z)\ast\lfloor(x,y,z)\rfloor^{-1}\in [0,1)^3$. Given the group rule \eqref{EqGroupRule}, it is not hard to check that
\begin{equation}\label{EqIntPart}\lfloor(x,y,z)\rfloor=(\lfloor x\rfloor, \lfloor y\rfloor, \lfloor z-(p^2-q^2)(x\lfloor y\rfloor-\lfloor x\rfloor y)\rfloor).\end{equation}

 Thus the map $\rho_0: X_\ast\to [0,1)^3$ given by  \begin{equation}\begin{aligned}\rho_0: (x,y,z)\Gamma\mapsto &(x,y,z)\ast\lfloor(x,y,z)\rfloor^{-1}\\
&=(\{x\},\{y\},\{ z-(p^2-q^2)(x\lfloor y\rfloor-\lfloor x\rfloor y)\})\end{aligned}\end{equation}
is bijective and provides a piecewise continuous parametrization of $X_\ast$ by $[0,1)^3$. Here $\{x\}$ stands for $x-\lfloor x\rfloor$, the fractional part of $x$.

If $\bfx_\ast\in X_\ast$ is represented by $(x,y,z)\in F$, then $T_\ast\bfx$ is represented by $\big(x+\alpha,y+\beta,z+(p^2-q^2)(\alpha y-\beta x)+H(x,y)\big)\in G_\ast$, and thus can also be represented by the element $$\begin{aligned}\Big(\{x+\alpha\},& \{y+\beta\},\Big\{z+H(x,y)\\
&+(p^2-q^2)\big((\alpha y-\beta x)-(x+\alpha)\lfloor y+\beta\rfloor+\lfloor x+\alpha\rfloor (y+\beta)\big)\Big\}\Big)\end{aligned}$$ in $[0,1)^3$.

If we identify $[0,1)^3$ with $\bT^3$ in the natural way, and let $\rho$ be the composition given by $X_\ast\overset{\rho_0}\to [0,1)^3\to\bT^3$, then $\rho$ is bijective and piecewise continuous. Moreover, the discussion above shows that $T_\ast$ is conjugate to the map 
\begin{equation}\label{EqTrivi}T_\ast':(x,y,z)\mapsto (x+\alpha,y+\beta, z+H'(x,y))\end{equation}
on $\bT^3$ by the piecewise continuous bijection $\rho$, where $H':\bT^2\to\bR$ is defined by \begin{equation}\begin{aligned}\label{EqTriviCocyle}H'(x,y)=&H(x,y)+(p^2-q^2)\big((\alpha y-\beta x)\\
&-(x+\alpha)\lfloor y+\beta\rfloor+\lfloor x+\alpha\rfloor (y+\beta)\big)\end{aligned}\end{equation} for $(x,y)\in[0,1)^2$ and regarded as a piecewise continuous map on $\bT^2$.

Therefore, in view of Theorem \ref{ThmFurst}, in order to obtain Proposition \ref{PropUE}, it suffices to show the following lemma:

\begin{lem}\label{LemMultiCocycle}For all $k\in \bN$, the equation 
\begin{equation}
\label{EqLemMultiCocycle}
R(x+\alpha,y+\beta)=R(x,y)+kH'(x,y)
\end{equation}
has no measurable solution $R:\bT^2\to \bT^1$ modulo $\rmm_{\bT^2}$. \end{lem}

Our approach to Lemma \ref{LemMultiCocycle} is inspired by \cite{F61}*{Lemma 2.2}

Notice first that, suppose $R(x,y)$ is such a solution, then the set $$\Lambda':=\{(x,y,z)\in\bT^3: kz=R(x,y)\},$$ which is a multi-valued graph over $\bT^2$,  is $T_\ast'$ invariant except for a $\rmm_{\bT^2}$-null set of $(x,y)$. Let $\Lambda=\rho^{-1}(\Lambda')\subset X_\ast$. Then $\Lambda$ intersects every $\bT^1$-fiber in exactly $k$ points that form a translate of $\frac1k\bZ/\bZ$. Moreover, $\Lambda$ is almost $T_\ast$-invariant, in the sense that there is a subset $A\subseteq\bT^2$ with $\rmm_{\bT^2}(A)=1$, such that if $\bfx_\ast\in\Lambda\cap\pi_{\bT^2}^{-1}(A)$, then $T_\ast\bfx_\ast\in\Lambda$.

\begin{lem}\label{LemJoiningIterate}For $\bfx_\ast=(x,y,z)\Gamma\in X_\ast$ and $n\in\bN$, $$T_\ast^n\bfx_\ast=\big(x+n\alpha,y+n\beta,H_n(x,y)\big)\bfx_\ast,$$ where $H_n(x,y)=\sum_{i=0}^{n-1}H(x+i\alpha,y+i\beta)$.\end{lem}

\begin{proof}The proof is the same as that of Lemma \ref{LemIterate}, using the new group rule $\ast$ in lieu of \eqref{EqGroupRule}.\end{proof}

Given $n\in\bN$, remark that $T_\ast^n$ is conjugate by $\rho$ to the $(T_\ast')^n$. Repeating the proof of \eqref{EqTrivi}, we can show similarly that 
\begin{equation}\label{EqIterTrivi}(T_\ast')^n(x,y,z)=(x+n\alpha,y+n\beta, z+H_n'(x,y))\end{equation}
on $\bT^3$, where $H_n':\bT^2\to\bR$ is defined by \begin{equation}\begin{aligned}\label{EqIterTriviCocyle}H_n'(x,y)=&H_n(x,y)+(p^2-q^2)\big((n\alpha y-n\beta x)\\
&-(x+n\alpha)\lfloor y+n\beta\rfloor+\lfloor x+n\alpha\rfloor (y+n\beta)\big).\end{aligned}\end{equation} for $(x,y)\in[0,1)^2$.

Because $\Lambda$ is almost $T_\ast$ invariant, it is also almost $T_\ast^n$ invariant. And $\Lambda'$ is almost $(T_\ast')^n$ invariant in the same sense, i.e. for a subset $A\subseteq\bT^2$ of full $\rmm_{\bT^2}$-measure, if $\bfx_\ast'\in\Lambda'\cap\pi_{\bT^2}^{-1}(A)$, then $(T_\ast')^n\bfx_\ast'\in\Lambda'$. This is equivalent to the statement that the equation
\begin{equation}
\label{EqIterMultiCocycle}
R(x+n\alpha,y+n\beta)=R(x,y)+kH_n'(x,y)
\end{equation}
holds for $\rmm_{\bT^2}$-almost all $(x,y)$.

\begin{proof}[Proof of Lemma \ref{LemMultiCocycle}]
Suppose $k\in\bN$ and $R:\bT^2\mapsto \bT^1$ is a measurable solution of \eqref{EqLemMultiCocycle}. Let \begin{equation}\label{EqPfLemMultiCocycle01}\delta_1=\dfrac{|k(p^2-q^2)d_1-k(p^2-q^2)\beta-\beta|}{24k(p^2+q^2)(L+|\alpha|+|\beta|)},\end{equation} and \begin{equation}\label{EqPfLemMultiCocycle02}\nu=\dfrac{6}{|k(p^2-q^2)d_1-k(p^2-q^2)\beta-\beta|},\end{equation} where  $d_1$ is the degree of $h$ in $x$ and $L$ is the Lipschitz constant of $h$. Note $\delta_1>0$ and $\nu<\infty$ because $p>q$, $k>0$, $p,q,k,d_1\in\bZ$ and $\beta\notin\bQ$.  By Luzin's theorem, we can find a compact subset $\Phi\subset\bT^2$ of measure greater than $1-\delta_1$ such that $R$ is continuous when restricted to $\Phi$. 

Choose $\delta_2\in(0,\min(\frac16,\delta_1))$ such that  if $(x,y),(x',y') \in \Phi$ and $\|(x,y) - (x',y')\|< \delta_2$, then $\|R(x,y)-R(x',y')\|< \frac{1}{3}$. We then fix $n\in\bN$ such that $\{n\alpha\}, \{n\beta\}\in(0,\delta_2)$ and $n >\nu$. Such integers $n$ exist because $T_0$ is minimal on $\bT^2$.

For $(x,y)\in(0,1-\delta_2)^2$, we have that $x+\{n\alpha\},y+\{n\beta\}\in(0,1)$ and $\lfloor x+n\alpha\rfloor=\lfloor n\alpha\rfloor$, $\lfloor y+n\beta\rfloor=\lfloor n\beta\rfloor$. Hence, \begin{equation}\begin{aligned}\label{EqPfLemMultiCocycle1}H_n'(x,y)=&H_n(x,y)+n(p^2-q^2)(\alpha y-\beta x)
-\lfloor n\beta\rfloor(x+n\alpha)\\
&+\lfloor n\alpha\rfloor (y+n\beta)\big),\ \forall(x,y)\in(\delta_2,1-\delta_2)^2.\end{aligned}\end{equation}

On the other hand, for $(x,y)\in\Phi\cap\big(\Phi-(n\alpha,n\beta)\big)$, $\|R(x+n\alpha,y+n\beta)-R(x,y)\|<\frac13$. So by \eqref{EqIterMultiCocycle}, 
\begin{equation}\label{EqPfLemMultiCocycle2}\|kH_n'(x,y)\|<\frac13,\ \forall(x,y)\in\Phi\cap\big(\Phi-(n\alpha,n\beta)\big).\end{equation}

Because $\rmm_{\bT^2}\big(\Phi\cap\big(\Phi-(n\alpha,n\beta)\big)\big)>1-2\delta_1$ and $\rmm_{\bT^2}\big((0,1-\delta_2)^2)>(1-\delta_2)^2>1-2\delta_2>1-2\delta_1$, by combining \eqref{EqPfLemMultiCocycle1} and \eqref{EqPfLemMultiCocycle2}, we know that 
\begin{equation}\label{EqPfLemMultiCocycle3}\Big\|kH_n(x,y)+nk(p^2-q^2)(\alpha y-\beta x)-\lfloor n\beta\rfloor(x+n\alpha)+\lfloor n\alpha\rfloor (y+n\beta)\Big\|<\frac13\end{equation} on a subset $\Phi_1\subset[0,1)^2$ with $\rmm_{\bR^2}(\Phi_1)>1-4\delta_1$, where $\rmm_{\bR^2}$ is the Lebesgue measure on $\bR^2$.

Fix a continuous lifting $\tH_n:\bR^2\to\bR^1$ of the function $H_n:\bT^2\to\bT^1$. Then \eqref{EqPfLemMultiCocycle3} actually asserts the continuous function $$F_n(x,y)=k\tH_n(x,y)+nk(p^2-q^2)(\alpha y-\beta x)-\lfloor n\beta\rfloor(x+n\alpha)+\lfloor n\alpha\rfloor (y+n\beta)$$ takes values in $\bigcup_{m\in\bZ}(m-\frac13,m+\frac13)$ on $\Phi_1$.

Because $h$ has degree $d_1$ in $x$, $h_j$ has degree $jd_1$ in $x$. Thus $H(x,y)=h_p(px,py)-h_q(qx,qy)$ has degree $(p^2-q^2)d_1$ in $x$. It in turn follows that $H_n(x,y)$ has degree $n(p^2-q^2)d_1$ in $x$. In consequence, for all $y\in\bR$, $\tH_n(1,y)-\tH_n(0,y)=n(p^2-q^2)d_1$ and thus
\begin{equation}\label{EqPfLemMultiCocycle4}F_n(1,y)-F_n(0,y)=nk(p^2-q^2)d_1-nk(p^2-q^2)\beta-\lfloor n\beta\rfloor.\end{equation}

By Fubini's Theorem, there exists $y_0\in[0,1)$ such that 
\begin{equation}\rmm_{\bR}(\{x\in[0,1]:(x,y_0)\notin\Phi_1\})<4\delta_1.\end{equation}
Because $F_n$ takes values in $\bigcup_{m\in\bZ}(m-\frac13,m+\frac13)$ on $\Phi_1$, the image  \begin{equation}\label{EqPfLemMultiCocycle5}\{F(x,y_0):x\in[0,1],(x,y_0)\notin\Phi_1\}\subset\bR\end{equation} has at least Lebesgue measure \begin{equation}\label{EqPfLemMultiCocycle6}\begin{aligned}&\frac13(|F_n(1,y)-F_n(0,y)|-1)\\
\geq&\frac13(|nk(p^2-q^2)d_1-nk(p^2-q^2)\beta-n\beta|-2)\\
\geq&n\cdot \frac13|k(p^2-q^2)d_1-k(p^2-q^2)\beta-\beta|-1.\end{aligned}\end{equation}

On the other hand, because $h$ is $L$-Lipschitz, $h_j$ is $jL$-Lipschitz and $H(x,y)$ is $(p^2+q^2)L$-Lipschitz. It in turn follows that $H_n(x,y)$ is $n(p^2+q^2)L$-Lipschitz and so is $\tH_n$. From \eqref{EqPfLemMultiCocycle3}, the Lipschitz constant of $F_n$ is at most $$nk(p^2+q^2)L+nk|p^2-q^2|(|\alpha|+|\beta|)+(|\alpha|+|\beta|)\leq nk(p^2+q^2)(L+|\alpha|+|\beta|).$$ So the image \eqref{EqPfLemMultiCocycle5} has at most Lebesgue measure  \begin{equation}\label{EqPfLemMultiCocycle7}\begin{aligned}&nk(p^2+q^2)(L+|\alpha|+|\beta|)\cdot4\delta_1\\
\leq& n\cdot \frac16|k(p^2-q^2)d_1-k(p^2-q^2)\beta-\beta|.\end{aligned}\end{equation}

Comparing \eqref{EqPfLemMultiCocycle6} with \eqref{EqPfLemMultiCocycle7} yields that $$n\cdot \frac16|k(p^2-q^2)d_1-k(p^2-q^2)\beta-\beta|\leq 1.$$ However, this contradicts the hypothesis that $n>\nu$. We arrive at a contradiction and the statement is proven.
\end{proof}

\begin{proof}[Proof of Theorem \ref{ThmMain}]Lemma \ref{LemMultiCocycle} and Theorem \ref{ThmFurst} imply Proposition \ref{PropUE}, contradicting Proposition \ref{PropNonUE}. Thus the standing hypothesis in Section \ref{SecNonUE} can not be true. In other words, \eqref{EqSarnak} must hold.\end{proof}

\end{document}